\documentclass[12pt]{amsart}
\usepackage{amsmath, amssymb, amsthm, bm, mathtools, enumitem, caption}
\usepackage{hyperref}

\usepackage[noabbrev,capitalize]{cleveref}
\usepackage{adjustbox}
\crefname{equation}{}{}
\usepackage{ytableau}
\usepackage{graphicx, tikz}
\usepackage[textheight=8.6in, textwidth=6.5in]{geometry}

\usepackage{microtype}

\newtheorem{theorem}{Theorem}[section]
\newtheorem{lemma}[theorem]{Lemma}

\newtheorem*{conjecture*}{Conjecture}
\newtheorem*{problem}{Problem}

\theoremstyle{definition}

\theoremstyle{remark}
\newtheorem*{remark}{Remark}

\newtheorem*{example}{Example}
\newtheorem*{examples}{Examples}

\numberwithin{equation}{section}

\DeclareMathOperator{\Tr}{Tr}

\newcommand{\C}{\mathbb C}

\newcommand{\Z}{\mathbb Z}

\title[Derivatives of theta functions as Traces of Partition Eisenstein series]{Derivatives of theta functions as Traces of Partition Eisenstein series}
\author{Tewodros Amdeberhan, Ken Ono \and Ajit Singh}
\address{Dept. of Mathematics, Tulane University, New Orleans, LA 70118}
\email{tamdeber@tulane.edu}
\address{Dept. of Mathematics, University of Virginia, Charlottesville, VA 22904}
\email{ko5wk@virginia.edu}
\email{ajit18@iitg.ac.in}

\begin{document}
\begin{abstract} 
In his ``lost notebook'', Ramanujan used iterated derivatives of two theta functions to define sequences of $q$-series  $\{U_{2t}(q)\}$ and $\{V_{2t}(q)\}$ that he claimed to be quasimodular. We give the first explicit proof of this claim by expressing them in terms of ``partition Eisenstein series'', extensions of the classical Eisenstein series $E_{2k}(q),$ defined by
$$\lambda=(1^{m_1}, 2^{m_2},\dots, n^{m_n}) \vdash n \ \ \ \ \ \longmapsto \ \ \ \ \ 
E_{\lambda}(q):= E_2(q)^{m_1} E_4(q)^{m_2}\cdots E_{2n}(q)^{m_n}.
$$
For functions $\phi  :  \mathcal{P}\mapsto \C$ on partitions,  the {\it weight $2n$ partition Eisenstein trace} is
$$
\Tr_n(\phi;q):=\sum_{\lambda \vdash n} \phi(\lambda)E_{\lambda}(q).
$$
For all $t$, we prove that $U_{2t}(q)=\Tr_t(\phi_U;q)$ and $V_{2t}(q)=\Tr_t(\phi_V;q),$  where
$\phi_U$ and $\phi_V$  are natural partition weights, giving the first explicit quasimodular formulas for these series.
 \end{abstract}

\maketitle
\section{Introduction and Statement of Results}

In his ``lost notebook'', Ramanujan considered the \cite[page 369]{Rama} two sequences of $q$-series:
\begin{equation}\label{U}
U_{2t}(q)=\frac{1^{2t+1}-3^{2t+1}q+5^{2t+1}q^3-7^{2t+1}q^6+\cdots} {1-3q+5q^3-7q^6+\cdots}
=\frac{\sum_{n\geq0}(-1)^n(2n+1)^{2t+1}q^{\frac{n(n+1)}2}}{\sum_{n\geq0}(-1)^n(2n+1)q^{\frac{n(n+1)}2}},
 \end{equation}
 \begin{equation}\label{V}
V_{2t}(q)=\frac{1^{2t}-5^{2t}q-7^{2t}q^2+11^{2t}q^5+\cdots} {1-q-q^2+q^5+\cdots}
=\frac{\sum_{n=-\infty}^{\infty}(-1)^n(6n+1)^{2t}q^{\frac{n(3n+1)}2}}{\sum_{n=-\infty}^{\infty}(-1)^n q^{\frac{n(3n+1)}2}},
\end{equation}
and he offered identities such as
\begin{displaymath}
\begin{split}
&U_0=1,\ \ U_2=E_2,\ \ U_4=\frac13(5E_2^2-2E_4),\ \ U_6=\frac19(35E_2^3-42E_2E_4+16E_6),\dots\\
&V_0=1,\ \ V_2=E_2,\ \ V_4=3E_2^2-2E_4,\  \qquad V_6=15E_2^3-30E_2E_4+16E_6,\dots
\end{split}
\end{displaymath}
where $E_2(q), E_4(q),$ and $E_6(q)$ are the usual Eisenstein series
\begin{displaymath}
\begin{split}
E_2:=1-24\sum_{n=1}^{\infty} \sigma_1(n)q^n,\ \ 
E_4:=1+240\sum_{n=1}^{\infty}\sigma_3(n)q^n, \ \  {\text {\rm and}}\ \ 
E_6:=1-504\sum_{n=1}^{\infty} \sigma_5(n)q^n,
\end{split}
\end{displaymath}
and where $\sigma_v(n):=\sum_{d\mid n}d^v.$
He made the following claim:
$$\text{``\emph{In general $U_{2t}$ and $V_{2t}$ are of the form
$\sum K_{\ell,m,n} \, E_2^{\ell}E_4^mE_6^n,$ where $\ell+2m+3n=t$."}}$$ 
Berndt, Chan, Liu, Yee, and Yesilyurt \cite{Berndt1, Berndt2} proved this claim using Ramanujan's identities \cite{Rama2}
\begin{equation}\label{diffeq}
D(E_2) =\frac{E_2^2-E_4}{12},\ \  \
D(E_4)=\frac{E_2E_4-E_6}{3},\ \ \ {\text {\rm and}}\ \ \
D(E_6)=\frac{E_2E_6-E_4^2}{2},
\end{equation}
where $D:=q\frac{d}{dq}.$ 
However, their results are not explicit. Indeed,
Andrews and Berndt (see p. 364 of \cite{AndrewsBerndt}) proclaim that ``...it seems extremely difficult to find a general formula for all $K_{\ell,m,n}.$'' 

We offer a solution to the general problem of obtaining the first explicit  formulas for $U_{2t}$ and $V_{2t}.$ 
We note that Ramanujan's claim is that $U_{2t}$ and $V_{2t}$  are weight $2t$ quasimodular forms, as  the ring of quasimodular forms  is the polynomial ring  (for example, see \cite{Zagier})
$$
\mathbb{C}[E_2,E_4,E_6]=\mathbb{C}[E_2, E_4, E_6, E_8, E_{10}, \dots],
$$
and so our goal is to obtain explicit formulas in terms of the classical sequence of Eisenstein series (for example, see Chapter 1 of \cite{CBMS})
\begin{equation}\label{Eisenstein}
E_{2k}(q):=1-\frac{4k}{B_{2k}}\sum_{n=1}^{\infty}\sigma_{2k-1}(n)q^n,
\end{equation}
where $B_{2k}$ is the $2k$th Bernoulli number.
We express Ramanujan's $q$-series as explicit ``traces of partition Eisenstein series.''  

As an important step towards this goal, we first derive generating functions for his series. In terms of
Dedekind's eta-function $\eta(q):=q^{\frac{1}{24}}\prod_{n=1}^{\infty}(1-q^n)$ and Jacobi-Kronecker quadratic characters, we have the following result.

\begin{theorem}\label{GenFunctionRamanujan} As a power series in $X$, the following are true.

\noindent (1)  If $\chi_{-4}(\cdot)=\left(\frac{-4}{\cdot}\right),$ then we have
$$\sum_{t\geq0}(-1)^t U_{2t}(q) \cdot  \frac{X^{2t+1}}{(2t+1)!}=
\frac{1}{2 \eta(q)^3}\cdot \sum_{n\in\mathbb{Z}} \chi_{-4}(n)q^{\frac{n^2}8}\sin(nX).
$$

\noindent (2) If $\chi_{12}(\cdot)=\left(\frac{12}{\cdot}\right),$ then we have
	$$\sum_{t\geq0}(-1)^t V_{2t}(q) \cdot  \frac{X^{2t}}{(2t)!}=
	\frac{1}{2 \eta(q)}\cdot \sum_{n\in\mathbb{Z}} \chi_{12}(n)q^{\frac{n^2}{24}}\cos(nX).
	$$
\end{theorem}

\begin{remark}Theorem~\ref{GenFunctionRamanujan} represents two special cases of Theorem~\ref{GenFunction}, which pertains to arbitrary theta functions. 
Using Theorem~\ref{GenFunctionRamanujan}, we obtain
 Theorem~\ref{PowerSeriesIdentities} that gives two further identities for these particular generating functions as infinite products in trigonometric functions.
\end{remark}

These generating functions shall offer the connection to traces of partition Eisenstein series. To make this precise, we recall that a {\it partition of $n$} is any nonincreasing sequence of positive integers 
$\lambda=(\lambda_1,\lambda_2,\dots, \lambda_s)$ that sum to $n$, denoted $\lambda\vdash n.$ Equivalently, we use the notation $\lambda=(1^{m_1},\dots,n^{m_n})\vdash n$, where $m_j$ is the multiplicity of $j.$ For such $\lambda,$ we define the weight $2n$ {\it partition Eisenstein series}\footnote{These $E_{\lambda}$ should not be confused with the partition Eisenstein series introduced by Just and Schneider \cite{JustSchneider}, which are semi-modular instead of quasimodular.}
\begin{equation}
\lambda=(1^{m_1}, 2^{m_2},\dots, n^{m_n}) \vdash n \ \ \ \ \ \longmapsto \ \ \ \ \  E_{\lambda}(q):= E_2(q)^{m_1}E_4(q)^{m_2}\cdots E_{2n}(q)^{m_n}.
\end{equation}
The Eisenstein series $E_{2k}(q)$ corresponds to the partition $\lambda=(k)$, as we have $E_{(k^1)}(q)=E_{2k}(q)^1.$

To define partition traces, suppose that $\phi: \mathcal{P}\mapsto \C$ is a function on partitions. For each positive integer $n$, its {\it partition Eisenstein trace}
is the weight $2n$ quasimodular form 
\begin{equation}\label{PartitionTrace}
\Tr_n(\phi;q):=\sum_{\lambda \vdash n} \phi(\lambda)E_{\lambda}(q).
\end{equation}
Such traces arise in recent work on MacMahon's sums-of-divisors $q$-series (see Thm. 1.4 of \cite{AOS}).

For partitions
$\lambda=(1^{m_1},\dots,n^{m_n})\vdash n$, we require the following functions:
\begin{equation}\label{phiu}
\phi_U(\lambda):=4^{n}(2n+1)! \cdot \prod_{k=1}^n \frac{1}{m_k !}\left(\frac{B_{2k}}{(2{\color{black}k}) \, (2k)!}\right)^{m_k} ,
\end{equation}
\begin{equation}\label{phiv}
\phi_V(\lambda):=4^n(2n)! \cdot\prod_{k=1}^n \frac1{m_k!} \left(\frac{(4^k-1)B_{2k}}{(2{\color{black}k})\, (2k)!}\right)^{m_k}.
\end{equation}
Ramanujan's series are weighted traces of partition Eisenstein series of these functions.

\begin{theorem}\label{MainTheorem} If $t$ is a positive integer, then the following are true.

\noindent
(1) We have that
$U_{2t}(q)=\Tr_t(\phi_U;q).$

\noindent
(2) We have that
$V_{2t}(q)=\Tr_t(\phi_V;q).$
\end{theorem}

\begin{examples} Here we offer examples of Theorem~\ref{MainTheorem}. \newline
\noindent
(1)  By direct calculation, we find for $t=3$ that
$$\phi_U((3^1))=16/9, \ \phi_U((1^1, 2^1))=-42/9,\ \ {\text {\rm and}}\ \  \phi_U((1^3))=35/9.
$$ This reproduces Ramanujan's identity
$$
\Tr_3(\phi_U;q)=\frac{1}{9}(16E_6-42E_2E_4+35E_2^3)=U_6.
$$

\noindent
(2) By direct calculation, we find for $t=4$ that
$$
\phi_V((4^1))=-272, \phi_V((1^1, 3^1))=448, \phi_V((2^2))=140, \ \phi_V((1^2, 2^1))=-420, \ {\text {\rm and}}\  
\phi_V((1^4))=105.
$$
Therefore, we have that
$$
\Tr_4(\phi_V;q)=-272 E_8+448 E_2E_6+140E_4^2-420E_2^2E_4+105E_2^4=V_8.
$$

\noindent
(3) Similar calculations give expressions for the weight 14 quasimodular forms $U_{14}$ and $V_{14}$
\begin{align*}
	U_{14}
	&=\frac1{405} (25025E_2^7 - 210210E_2^5E_4 + 420420E_2^3E_4^2 - 168168E_2E_4^3 + 400400E_2^4E_6 \\
	&- 960960E_2^2E_4E_6 + 192192E_4^2E_6 + 366080E_2E_6^2 - 720720E_2^3E_8 + 864864E_2E_4E_8 \\
	&- 329472E_6E_8 + 1048320E_2^2E_{10} - 419328E_4E_{10} - 1061376E_2E_{12} + 552960E_{14}), \\
	V_{14}
	&=135135E_2^7 - 1891890E_2^5E_4 + 6306300E_2^3E_4^2 - 4204200E_2E_4^3 + 5045040E_2^4E_6 \\
	& - 20180160E_2^2E_4E_6 + 6726720E_4^2E_6 + 10762752E_2E_6^2 - 12252240E_2^3E_8 \\
	& + 24504480E_2E_4E_8 - 13069056E_6E_8 + 23831808E_2^2E_{10} - 15887872E_4E_{10} \\
	& - 32195072E_2E_{12} + 22368256E_{14}.
\end{align*}
Then, using the identities
$$E_{14}=E_2^4E_6, \,\,\,\, E_{12}=\frac{250}{691}E_6^2+\frac{441}{691}E_4^3, \,\,\,\, E_{10}=E_4E_6, \,\,\,\, E_8=E_4^2$$
we get
\begin{align*}
	U_{14}
	&=\frac1{27}(-3648E_4^2E_6 - 17920E_2E_6^2 + 19320E_2E_4^3 - 300300E_2^3E_4^2 \\
	& \qquad + 400400E_2^4E_6 - 210210E_2^5E_4 + 87360E_2^2E_4E_6 + 25025E_2^7), \\
	V_{14}
	&=138048E_4^2E_6 - 885248E_2E_6^2 - 246792E_2E_4^3 -  5945940E_2^3E_4^2 \\
	& \qquad + 5045040E_2^4E_6 - 1891890E_2^5E_4 + 3651648E_2^2E_4E_6 + 135135E_2^7.
\end{align*}
\end{examples}

\begin{remark} The coefficients $\phi_V(\lambda)$ are always integers. The first author and other collaborators have found a combinatorial interpretation and a proof this fact, which will appear in a forthcoming work.
\end{remark}

In view of these results, it is natural to  pose the following problem.

\begin{problem} Determine and characterize further functions $\phi: \mathcal{P}\mapsto \C$ for which
$\{\Tr_t(\phi;q)\}$ is a natural and rich family of weight $2t$ quasimodular forms.
\end{problem}

To prove these results, we make use of the Jacobi Triple Product identity, special $q$-series, exponential generating functions for Bernoulli numbers, and properties of P\'olya's cycle index polynomials.
In Section~\ref{Section2} we derive a general result for $q$-series of the form (\ref{U}) and (\ref{V}) (see Theorem~\ref{GenFunction}), which gives Theorem~\ref{GenFunctionRamanujan} as special cases.
In Section~\ref{Section3} we prove
Theorem~\ref{MainTheorem} using these results and properties of P\'olya's cycle index polynomials.

\section*{Acknowledgements}
\noindent The authors thank the referee, George Andrews and Wei-Lun Tsai for comments on earlier drafts of this paper.
 The second author thanks the Thomas Jefferson Fund and the NSF
(DMS-2002265 and DMS-2055118). The third author is grateful for the support of a Fulbright Nehru Postdoctoral Fellowship.

\section{Generating functions for Ramanujan-type $q$-series}\label{Section2}

Theorem~\ref{GenFunctionRamanujan} gives two special cases of general generating functions associated to formal theta functions for Dirichlet characters. If $\chi$ modulo $N$ is a Dirichlet character, then let
\begin{equation}
\Theta(\chi;q):=\sum_{n=1}^{\infty} \chi(n)n^{a_{\chi}}q^{n^2},
\end{equation}
where we let
\begin{equation}
a_{\chi}:=\begin{cases} 0 \ \ \ \ \ &{\text {\rm if $\chi$ is even}},\\
1 \ \ \ \ \ &{\text {\rm if $\chi$ is odd.}}
\end{cases}
\end{equation}
Then,  in analogy with Ramanujan's $U_{2t}$ and $V_{2t}$ (see (\ref{U}) and (\ref{V})), we let
\begin{equation}\label{Rdef}
R_{2t}(\chi;q):=\frac{D^t\left(\Theta(\chi;q)\right)}{\Theta(\chi;q)}=
\frac{\sum_{n=1}^{\infty} \chi(n)n^{2t+a_{\chi}}q^{n^2}}{\Theta(\chi;q)}.
\end{equation}

\begin{lemma}\label{GenFunction} Assuming the notation above, as a power series in $X$ we have
$$
\sum_{t=0}^{\infty}(-1)^t R_{2t}(\chi;q)\cdot \frac{X^{2t+1}}{(2t+1)!}=
\frac{1}{2i\Theta(\chi;q)}\sum_{n\in \Z} \chi(n)q^{n^2}n^{a_{\chi}-1}\cdot e^{inX}.
$$
\end{lemma}

\begin{remark} Theorem~\ref{GenFunction} holds for periodic functions $\chi : \Z\rightarrow \C$ that are either even or odd.
\end{remark}
\begin{proof}
By direct calculation, we have that
$$
\frac{1}{2i}\sum_{n\in \Z} \chi(n)q^{n^2}n^{a_\chi-1}\cdot e^{inX}=
\frac{1}{2i} \sum_{n=1}^{\infty}q^{n^2}n^{a_\chi-1}\left(\chi(n)e^{inX}+(-1)^{a_\chi-1}\chi(-n)e^{-inX}\right).
$$
For all $\chi,$ we have that $(-1)^{a_\chi-1}\chi(-n)=-\chi(n),$ and so this reduces to
$$
\frac{1}{2i}\sum_{n=1}^{\infty}\chi(n)q^{n^2}n^{a_\chi-1}\left(e^{inX}-e^{-inX}\right)=
\sum_{n=1}^{\infty}\chi(n)q^{n^2}n^{a_\chi-1}\sin(nX).
$$
Using the Taylor series for $\sin(nX)$, this gives (after change of summation) 
\begin{displaymath}
\begin{split}\frac{1}{2i}\sum_{n\in \Z} \chi(n)q^{n^2}n^{a_\chi-1}\cdot e^{inX}&=
\sum_{n=1}^{\infty}\chi(n)n^{a_\chi-1}q^{n^2}\sum_{t=0}^{\infty}(-1)^t\cdot \frac{(nX)^{2t+1}}{(2t+1)!}\\
&=\sum_{t=0}^{\infty}(-1)^t\cdot \frac{X^{2t+1}}{(2t+1)!}\cdot D^t(\Theta(\chi;q)).
\end{split}
\end{displaymath}
Thanks to (\ref{Rdef}), we obtain the claimed generating function by dividing through by $\Theta(\chi;q).$
\end{proof}

\begin{proof}[Proof of Theorem~\ref{GenFunctionRamanujan}] To prove claim (1), we consider ${\color{black}\chi_{-4}(n)}:= \left(\frac{-4}n\right),$ which is the only odd character modulo $4$. In this case we have $a_{\color{black}\chi_{-4}}=1$, and so Theorem~\ref{GenFunction} gives
$$
\sum_{t=0}^{\infty}(-1)^t R_{2t}({\color{black}\chi_{-4}};q)\cdot \frac{X^{2t+1}}{(2t+1)!}=
\frac{1}{2i\Theta({\color{black}\chi_{-4}};q)}\sum_{n\in \Z} \chi_{-4}(n)q^{n^2}\cdot e^{inX},
$$
Furthermore, Jacobi's classical identity (for example, see  p. 17 of \cite{CBMS}) implies that
$$
\eta(q^8)^3 = \Theta(\chi_{-4};q)=\sum_{n=0}^{\infty} (-1)^n (2n+1)q^{(2n+1)^2}.
$$
Therefore, we have that
$$
\sum_{t=0}^{\infty}(-1)^t R_{2t}(\chi_{-4};q)\cdot \frac{X^{2t+1}}{(2t+1)!}=
\frac{1}{2i\eta(q^8)^3}\sum_{n\in \Z} \chi_{-4}(n)q^{n^2}\cdot e^{inX}.
$$
Claim (1) follows by letting $q\rightarrow q^{\frac{1}{8}},$ replacing the complex exponential in terms of trigonometric functions, followed by taking the real part.

To prove claim (2), we note that Euler's Pentagonal Number Theorem (see p. 17 of \cite{CBMS})  implies that
$$
\eta(q^{24})=\sum_{n=1}^{\infty} \chi_{12}(n)q^{n^2},
$$
where $\chi_{12}(n)=\left(\frac{12}n\right)$ is the unique primitive character with conductor $12.$ Therefore, $a_{\chi}=0$, and so Theorem~\ref{GenFunction} gives 
$$
	\sum_{t=0}^{\infty}(-1)^t R_{2t}(\chi_{12};q) \cdot \frac{X^{2t+1}}{(2t+1)!}=
	\frac{1}{2i\eta(q^{24})}\sum_{n\in \Z} \chi_{12}(n)\frac{q^{n^2}}{n}\cdot e^{inX}.
	$$
Claim (2) follows by letting $q\rightarrow 	q^{\frac{1}{24}},$ then differentiating in $X$, followed by taking the real part as in (1).
\end{proof}

\section{Proof of Theorem~\ref{MainTheorem}}\label{Section3}
Here we prove Theorem~\ref{MainTheorem} using the generating functions in Theorem~\ref{GenFunctionRamanujan}. We apply P\'olya's cycle index polynomials and the exponential generating function for Bernoulli numbers. 

\subsection{Bernoulli numbers} Here we recall  a convenient generating function for Bernoulli numbers and we refer the reader to \cite[1.518.1]{GR}.
If $\text{\rm sinc}(X):=\sin X/X$, then we have
\begin{align} \label{sinc}
\frac1{\text{\rm sinc}(X)}&=\exp\left(-\sum_{k\geq1}\frac{(-4)^kB_{2k}}{(2k)(2k)!}\cdot  X^{2k}\right).
\end{align}

\subsection{P\'olya's cycle index polynomials}

We require P\'olya's cycle index polynomials in the case of symmetric groups (for example, see \cite{Stanley}). Namely, recall that given a partition $\lambda=(\lambda_1,\dots,\lambda_{\ell(\lambda)})\vdash t$ or $(1^{m_1},\dots,t^{m_t})\vdash t$, we have that the number of permutations in $\mathfrak{S}_t$ of cycle type $\lambda$ is $t!/z_{\lambda}$, where $z_{\lambda}:=1^{m_1}\cdots t^{m_t}m_1!\cdots m_t!$. The {\it cycle index polynomial} for the symmetric group $\mathfrak{S}_t$ is given by
\begin{align} \label{CIF}
Z(\mathfrak{S}_t)&=\sum_{\lambda\vdash t}\frac1{z_{\lambda}}\prod_{i=1}^{\ell(\lambda)}x_{\lambda_i}
=\sum_{\lambda\vdash t}\prod_{k=1}^t\frac1{m_k!}\left(\frac{x_k}{k}\right)^{m_k}.
\end{align}
We require the following well known generating function in $t$-aspect.
\begin{lemma}[Example 5.2.10 of \cite{Stanley}]\label{PolyaGenFunction} As a power series in $y$, the generating function for the cycle index polynomials satisfies
$$\sum_{t\geq0}Z(\mathfrak{S}_t)\,y^t=\exp\left(\sum_{k\geq1}x_k\frac{y^k}k\right).$$
\end{lemma}

\begin{example}
 Here are the first few examples of P\'olya's cycle index polynomials:
$$Z(\mathfrak{S}_1)=x_1, Z(\mathfrak{S}_2)=\frac1{2!}(x_1^2+x_2), Z(\mathfrak{S}_3)=\frac1{3!}(x_1^3+3x_1x_2+2x_3).
$$
\end{example}

\subsection{Some power series identities}

We begin with formulas for the infinite series factors of the generating functions in Theorem~\ref{GenFunctionRamanujan}.

\begin{lemma} \label{jacobi_plus} As a power series in $X$, the following are true.

\noindent
(1) We have that
$$\frac{q^{-\frac{1}{8}}}{2}\cdot\sum_{n\in\mathbb{Z}} \chi_{-4} \cdot q^{\frac{n^2}8}\sin(nX)
=\sin X \cdot \prod_{j\geq1}(1-q^j)(1-2\cos(2X)q^j+q^{2j}).
$$

\noindent
(2) We have that
\begin{displaymath}
\begin{split}
\frac{q^{-\frac{1}{24}}}{2}&\sum_{n\in\mathbb{Z}} \chi_{12}\cdot q^{\frac{n^2}{24}} \cos(nX)\\
&\ \ \ \ = \cos X\prod_{n\geq1}(1-q^n)(1+2\cos(2X)q^n+q^{2n})(1-2\cos(4X)q^{2n-1}+q^{4n-2}).
\end{split}
\end{displaymath}
\end{lemma}
\begin{proof} Both claims follow from the Jacobi Triple Product Identity (see Th. 2.8 of \cite{Andrews})
\begin{equation}\label{JTP}
\sum_{n\in\mathbb{Z}}(-1)^nq^{\frac{n^2}2}z^n=\prod_{j\geq1}(1-q^j)(1-q^{j-\frac{1}2}z)(1-q^{j-\frac{1}2}z^{-1}).
\end{equation}
To prove (1), we make the substitutions $2i\sin X=e^{iX}(1-e^{-2iX})$ and $z=q^{\frac12}e^{2iX}$ to obtain
\begin{align*}
\frac1{2i} \sum_{n\in\mathbb{Z}}(-1)^nq^{\frac{n^2+n}{2}}e^{(2n+1)iX}=\sin X\prod_{j\geq1}(1-q^j)(1-q^je^{2iX})(1-q^je^{-2iX}) .
\end{align*}
To obtain claim (1), we note the following simple reformulation
\begin{align*}
	\frac1{2i} \sum_{n\in\mathbb{Z}}(-1)^nq^{\frac{n^2+n}{2}}e^{(2n+1)iX}
=\frac{q^{-\frac{1}{8}}}{2i}\cdot\sum_{n\in\mathbb{Z}} \chi_{-4} \cdot q^{\frac{n^2}8}e^{inX},
\end{align*}
and then take the real part of both sides.

We now turn to claim (2). Thanks to \cite[(1.1)]{QPI}, we obtain
$$\sum_{n\in\mathbb{Z}} q^{n(3n+1)}(z^{3n}-z^{-3n-1})
=\prod_{n\geq1}(1-q^{2n})(1-zq^{2n})\left(1-\frac{q^{2n-2}}{z}\right)(1-z^2q^{4n-2})\left(1-\frac{q^{4n-2}}{z^2}\right).$$

By replacing $q\rightarrow q^{\frac{1}{2}}, z\rightarrow -z^2$, factoring out $1+z^{-2}$ and multiplying through by $z$, we get
$$\sum_{n\in\mathbb{Z}} (-1)^nq^{\frac{n(3n+1)}2}\frac{z^{6n+1}+z^{-6n-1}}{z+z^{-1}}
=\prod_{n\geq1}(1-q^n)(1+z^2q^n)\left(1+\frac{q^n}{z^2}\right)(1-z^4q^{2n-1})\left(1-\frac{q^{2n-1}}{z^4}\right).$$
After letting $z=e^{iX},$ we pair up conjugate terms to get 
\begin{displaymath}
\begin{split}
&\sum_{n\in\mathbb{Z}} (-1)^nq^{\frac{n(3n+1)}2}\cos(6n+1)X\\
&\hskip.5in =\cos X \prod_{n\geq1}(1-q^n)(1+2\cos(2X)q^n+q^{2n})(1-2\cos(4X)q^{2n-1}+q^{4n-2}).
\end{split}
\end{displaymath}
 The left hand side of the expression above equals the infinite sum in Lemma \ref{jacobi_plus} (2).
\end{proof}

To prove Theorem~\ref{MainTheorem}, we also require the following power series identities that give reformulations of the generating functions for $U_{2t}$ and $V_{2t}.$

\begin{theorem}\label{PowerSeriesIdentities} The following identities are true.

\noindent
(1) As power series in $X$, we have
$$
\sum_{t\geq0} (-1)^t U_{2t}(q)\cdot \frac{X^{2t+1}}{(2t+1)!}=
\sin X \cdot \prod_{j\geq1}\left[1+\frac{4(\sin^2X)q^j}{(1-q^j)^2}\right].
$$

\noindent
(2) As power series in $X$, we have
$$
\sum_{t\geq0} (-1)^t V_{2t}(q)\cdot \frac{X^{2t}}{(2t)!}=
\cos X \cdot
\prod_{j\geq1}\left[1-\frac{4(\sin^2X)q^j}{(1+q^j)^2}\right]\left[1+\frac{4(\sin^22X)q^{2j-1}}{(1-q^{2j-1})^2}\right].
$$
\end{theorem}

\begin{proof}
We first prove claim (1). By combining Theorem \ref{GenFunctionRamanujan} (1) and Lemma \ref{jacobi_plus} (1), we obtain 
\begin{align*}
 \sum_{t\geq0} (-1)^t U_{2t}(q)\cdot \frac{X^{2t+1}}{(2t+1)!}=\sin X\cdot \prod_{j\geq1}\frac{(1-2\cos(2X)q^j+q^{2j})}{(1-q^j)^2}.
\end{align*}
A straightforward algebraic manipulation with $-2\cos(2X)=-2+4\sin^2X$ yields
$$\sum_{t\geq0} (-1)^t U_{2t}(q)\cdot \frac{X^{2t+1}}{(2t+1)!}=\sin X \cdot \prod_{j\geq1}\left[1+\frac{4(\sin^2X)q^j}{(1-q^j)^2}\right].
$$

Now we turn to claim (2). By combining Theorem \ref{GenFunctionRamanujan} (2) and Lemma \ref{jacobi_plus} (2), we obtain 
 \begin{align*}
 &\sum_{t\geq0} (-1)^t V_{2t}(q)\cdot \frac{X^{2t}}{(2t)!}
=\cos X\prod_{n\geq1}(1+2\cos(2X)q^n+q^{2n})(1-2\cos(4X)q^{2n-1}+q^{4n-2}) \\
&\ \hskip.5in =\cos X \prod_{k\geq1}(1+q^k)^2(1-q^{2k-1})^2
\prod_{j\geq1}\left[1-\frac{4(\sin^2X)q^j}{(1+q^j)^2}\right]\left[1+\frac{4(\sin^22X)q^{2j-1}}{(1-q^{2j-1})^2}\right].
\end{align*} 
The proof now follows from the simple identity
$$\prod_{k=1}^{\infty}(1+q^k)=\prod_{k=1}^{\infty}\frac1{1-q^{2k-1}}.
$$
\end{proof}

\subsection{Proof of Theorem~\ref{MainTheorem}}
For each positive odd integer $j$, we consider the Lambert series
\begin{equation}\label{Sj}
\mathbf{S}_j(q):=\sum_{m\geq1}\frac{m^jq^m}{1-q^m} 
= \frac{B_{j+1}}{2(j+1)} - \frac{B_{j+1}}{2(j+1)}E_{j+1}(q).
\end{equation}
This expression in terms of the $E_{j+1}(q)$ follows from (\ref{Eisenstein}).
The proof of Theorem~\ref{MainTheorem} boils down to deriving expressions for the power series in Theorem~\ref{PowerSeriesIdentities} in terms of the $\mathbf{S}_j(q).$

\begin{proof}[Proof of Theorem~\ref{MainTheorem}]
We first prove claim (1) regarding Ramanujan's $U_{2t}$ series. The key fact underlying the proof is the following power series identity.
\begin{equation}\label{keyidentity}
\sum_{t\geq0} (-1)^t U_{2t}(q)\cdot \frac{X^{2t+1}}{(2t+1)!}
=\sin X\cdot \exp\left(-2\sum_{r\geq1}\frac{\mathbf{S}_{2r-1}(q)}{(2r)!} (-4X^2)^r \right).
\end{equation}
Thanks to Theorem~\ref{PowerSeriesIdentities} (1), this identity will follow from
\begin{equation}\label{identity1}
\exp\left(-2\sum_{r\geq1}\frac{\mathbf{S}_{2r-1}(q)}{(2r)!} (-4X^2)^r\right)=
\prod_{j\geq1} \left[1+\frac{4(\sin^2X)q^j}{(1-q^j)^2}\right].
\end{equation}
To establish (\ref{identity1}),  we compute the following double-sum in two different ways. First, we use the Taylor expansion of $\cos(y)$ and then interchange the order of summation to get
$$\sum_{j,k\geq1}\frac{q^{kj}\cos(2kX)}k
=\sum_{r\geq0}\frac{(-4X^2)^r}{(2r)!}\sum_{k\geq1}k^{2r-1}\sum_{j\geq1}q^{kj}.$$
By combining the geometric series and the Taylor series for $\log(1-Y)$ with  \eqref{Sj}, we obtain
\begin{equation}\label{eq1}
\sum_{j,k\geq1}\frac{q^{kj}\cos(2kX)}k	=\sum_{r\geq0}\frac{(-4X^2)^r}{(2r)!}\sum_{k\geq1} \frac{k^{2r-1}q^k}{1-q^k}
	=-\log(q)_{\infty}+\sum_{r\geq1}\frac{\mathbf{S}_{2r-1}(q)}{(2r)!} (-4X^2)^r,
\end{equation}
where $(q)_{\infty}:=\prod_{n=1}^{\infty}(1-q^n)$ is the $q$-Pochhammer symbol.

On the other hand, using $2\cos\theta=e^{i\theta}+e^{-i\theta}$ and Taylor expansion of $\log(1-Y)$, we find that	
\begin{displaymath}
\begin{split}\sum_{j,k\geq1}\frac{q^{kj}\cos(2kX)}k
	&=\frac12\sum_{j\geq1} \left(\sum_{k\geq1} \frac{(e^{2iX}q^j)^k}k+\sum_{k\geq1} \frac{(e^{-2iX}q^j)^k}k\right)\\
	&=-\frac12\sum_{j\geq1} \log\left[1-2(\cos 2X )q^j+q^{2j}\right].
\end{split}
\end{displaymath}
After straightforward algebraic manipulation, we get
	\begin{equation}\label{eq2}
	\sum_{j,k\geq1}\frac{q^{kj}\cos(2kX)}k
		=-\log(q)_{\infty}-\frac12\log\left(\prod_{j\geq1} \left[1+\frac{4(\sin^2X)q^j}{(1-q^j)^2}\right]\right).
\end{equation} 
	Identity (\ref{identity1}) follows by comparing \eqref{eq1} and \eqref{eq2}, and in turn confirms (\ref{keyidentity}).

We now investigate the exponential series in (\ref{keyidentity}).	Thanks to 
 (\ref{Sj}), followed by an application of (\ref{sinc}), we obtain
\begin{align*}
&\exp\left(-2\sum_{k\geq1}\frac{(-4)^k\mathbf{S}_{2k-1}(q)}{(2k)!} \,X^{2k} \right) \\
&\hskip.4in =\exp\left(-\sum_{k\geq1}\frac{(-4)^kB_{2k}}{(2k)(2k)!}\,X^{2k} \right) \cdot
\exp\left(\sum_{k\geq1}\frac{B_{2k}\cdot E_{2k}(q)}{(2k)(2k)!} \,(-4X^2)^k\right) \\
&\hskip.4in =\frac{X}{\sin X}\cdot \exp\left(\sum_{k\geq1}\frac{B_{2k}\cdot E_{2k}(q)}{2\,(2k)!} \,\frac{(-4X^2)^k}k \right). 
\end{align*}
We recognize this last expression in the context of P\'olya's cycle index polynomials. Namely,
Lemma~\ref{PolyaGenFunction} gives the identity (here $\lambda=(1^{m_1}\dots t^{m_t}\vdash t)$)
$$\exp\left(\sum_{k\geq1}Y_k\,\frac{w^k}k\right)
=\sum_{t\geq0}\left(\sum_{\lambda\vdash t}  
\prod_{k=1}^t \frac1{m_k!} \left(\frac{Y_k}{k}\right)^{m_k} 
\right)w^t,
$$
which we apply with $Y_k=\frac{B_{2k}\cdot E_{2k}(q)}{2(2k)!}$ and $w=-4X^2.$ This gives
\begin{align*}
&\sum_{t\geq0} (-1)^t U_{2t}(q)\frac{X^{2t+1}}{(2t+1)!}
=\sin X\cdot \exp\left(-2\sum_{k\geq1}\frac{\mathbf{S}_{2k-1}(q)}{(2k)!} \,(-4X^2)^{k}\right) \\
&\hskip.7in =\sin X\cdot\frac{X}{\sin X}\cdot
\sum_{t\geq0} \left(\sum_{\lambda\vdash t} \prod_{k=1}^t \frac1{m_k!}
\left(\frac{B_{2k}\cdot E_{2k}(q)}{(2k)(2k)!}\right)^{m_k}\right)(-4X^2)^t.
\end{align*}
By comparing the coefficients of $X^{2t+1}$, we find that
$$ (-1)^t \frac{U_{2t}(q)}{(2t+1)!}=(-4)^t\sum_{\lambda\vdash t}  
\prod_{k=1}^t \frac1{m_k!} \left(\frac{B_{2k}\cdot E_{2k}(q)}{{(2\color{black}k)}(2k)!}\right)^{m_k},$$
which in turn, thanks to
(\ref{phiu}), proves Theorem~\ref{MainTheorem} (1).

We now turn to claim (2) regarding Ramanujan's $V_{2t}$ whose proof is analogous to the proof of  (1). The main difference follows from the need for the generalized Lambert series
\begin{equation}\label{AA}
\mathbf{A}_{2r-1}(q):=\sum_{k\geq1}\frac{(-1)^{k-1}k^{2r-1}q^k}{1-q^k} =
\mathbf{S}_{2r-1}(q) - 4^r\mathbf{S}_{2r-1}(q^2).
\end{equation}
The expression in $\mathbf{S}_{2r-1}$ is straightforward.
For the sake of brevity, we note that calculations analogous to the proof of (\ref{identity1}) gives the identity
$$
\prod_{n\geq1}\left[1-\frac{4(\sin^2X)q^n}{(1+q^n)^2}\right]=
\exp\left(2\sum_{r\geq1}\frac{\mathbf{A}_{2r-1}(q)(-4X^2)^r}{(2r)!} \right),
$$
as well as 
\begin{displaymath}
\begin{split}
&\prod_{n\geq1}\left[1+\frac{4(\sin^22X)q^{2n-1}}{(1-q^{2n-1})^2}\right]
     =\prod_{n\geq1}\left[1+\frac{4(\sin^22X)q^n}{(1-q^n)^2}\right] \cdot
     \prod_{n\geq1}\left[1+\frac{4(\sin^22X)q^{2n}}{(1-q^{2n})^2}\right]^{-1} \\
		&\hskip.7in=\exp\left(-2\sum_{r\geq1}\frac{\mathbf{S}_{2r-1}(q)(-16X^2)^r}{(2r)!} \right)
		\exp\left(2\sum_{r\geq1}\frac{\mathbf{S}_{2r-1}(q^2)(-16X^2)^r}{(2r)!} \right).
\end{split}
\end{displaymath}
Therefore, combining these two expressions with (\ref{AA}), Theorem~\ref{PowerSeriesIdentities} (2) gives
	\begin{align*}
		&\sum_{t\geq0} (-1)^t V_{2t}(q) \cdot \frac{X^{2t}}{(2t)!} \\
		&\hskip.7in = \cos X\cdot \exp\left( -2\sum_{r\geq1}\frac{\mathbf{S}_{2r-1}(q)}{(2r)!} (-4(2X)^2)^r\right) 
		  \exp\left( 2\sum_{r\geq1}\frac{\mathbf{S}_{2r-1}(q)}{(2r)!} (-4X^2)^r\right).
	\end{align*}
As in the proof of (1), we recognize generating functions for P\'olya's cycle index polynomials. Namely, by applying	
(\ref{sinc}) and Lemma~\ref{PolyaGenFunction} we obtain
	\begin{align*} 
		\sum_{t\geq0} \frac{(-1)^t V_{2t}(q) X^{2t}}{(2t)!} 
		&=\cos X\cdot \frac{2X}{\sin(2X)} \cdot 
		\exp\left(\sum_{k\geq1}\frac{4^kB_{2k}\cdot E_{2k}(q)}{2\, (2k)!}\frac{(-4X^2)^k}k\right)    \\
		 &  \qquad \qquad \qquad \times \frac{\sin X}{X} \cdot 
                                       \exp\left(\sum_{k\geq1}\frac{-B_{2k}\cdot E_{2k}(q)}{2\, (2k)!}\frac{(-4X^2)^k}k\right)    \\
		&=\exp\left(\sum_{k\geq1}\frac{(4^k-1)B_{2k}\cdot E_{2k}(q)}{2\, (2k)!}\frac{(-4X^2)^k}k\right) \\
                    &=\sum_{t\geq0}\left(\sum_{\lambda\vdash t}\prod_{k=1}^t\frac1{m_k!}\left(\frac{(4^k-1)B_{2k}\cdot E_{2k}(q)}{(2{\color{black}k}) (2k)!}\right)^{m_k}\right) (-4X^2)^t.
	\end{align*}
	By comparing the coefficients of $X^{2t}$, we deduce that 
	$$V_{2t}(q)=4^t\, (2t)! \sum_{\lambda\vdash t} \prod_{k=1}^t \frac1{m_k!} \left(\frac{(4^k-1)B_{2k}\cdot E_{2k}(q)}{(2{\color{black}k}) (2k)!}\right)^{m_k},
	$$
	which thanks to (\ref{phiv}) completes the proof of Theorem~\ref{MainTheorem} (2).

\end{proof}

\end{document}